\documentclass{birkjour}
 \newtheorem{thm}{Theorem}[section]
 
 \newtheorem{lem}[thm]{Lemma}
 
 \theoremstyle{definition}
 
 \theoremstyle{remark}
 \newtheorem{rem}[thm]{Remark}
 
 \numberwithin{equation}{section}

\newcommand{\nr}{{|\!|}}

\newcommand{\1}{H^1(\mathbb{R}^N)}
\newcommand{\s}{\int_{\mathbb{R}^N}}

\newcommand{\stcn}{\xrightarrow[n\to +\infty]{} }
\newcommand{\stcr}{\xrightarrow[r\to +\infty]{} }
\begin{document}

%
%
%
%
%
%
%
%
%

\title[Analysis of a vectorial variational problem]
 {Minimizers of a class of constrained vectorial variational problems: Part I.}

\author[H. Hajaiej]{Hichem Hajaiej}

\address{%
Department of Mathematics, College of Science\br
King Saud University,\br
P.O. Box 2455\br
Riyadh 11451\br
Saudi Arabia.}

\email{hhajaiej@ksu.edu.sa}

\author[P.A. Markowich]{Peter A. Markowich}
\address{Division of Math \& Computer Sc \& Eng\br
King Abdullah University of Science and Technology\br
Thuwal 23955-6900\br
Saudi Arabia}
\email{peter.markowich@kaust.edu.sa}
\author[S. Trabelsi]{Saber Trabelsi}
\address{Division of Math \& Computer Sc \& Eng\br
King Abdullah University of Science and Technology\br
Thuwal 23955-6900\br
Saudi Arabia}
\email{saber.trabelsi@kaust.edu.sa}
\subjclass{12345, 54321}


\date{September 5, 2013}
\dedicatory{}

\begin{abstract}
In this paper, we prove the existence of minimizers of a class of multi-constrained
variational problems. We consider systems involving a nonlinearity that does not
satisfy compactness, monotonicity, neither symmetry properties. Our approach hinges on the concentration-compactness approach. In the second  part, we will treat orthogonal constrained problems for another class of integrands using density matrices method.
\end{abstract}

\maketitle

\section{\bf Introduction}
\setcounter{equation}{0}
Let $c > 0$ a prescribed constant,  we consider the
following minimization problem
\[
 \mathcal I_{c}=\inf\{\mathcal J(\vec{u}),\qquad \vec{u} \in \mathcal S_c\},\eqno{(\mathcal I)}
 \]
where $\mathcal J $ models an energy functional given as follows
\[\mathcal J(\vec{u}) = \s\left(\frac12\;|\nabla\vec{u}|^2 - 
 F(x,\vec{u})\right) dx,\]
with  $\vec{u} = (u_1,...,u_m) \in  \times^m\1:=\vec{H}^1$  for all integer $m\geq 1$ and a
Carath\'eodory function $F$ satisfying few assumptions listed below. The set $\mathcal S_c$ is given by 
\[
\mathcal S_c = \{\vec{u} = (u_1,...,u_m) \in \times^mH^1,\qquad
 \sum_{i=1}^m \s u^2_i = c^2\}.
 \]
Formally (rigourously under some regularity assumptions on  $F$), solutions
 of $(\mathcal I)$ satisfy 
 \[\left\{ \begin{array}{l}
 \Delta u_1 + \partial_1 F(x, u_1,...,u_m) + \lambda u_1 = 0,\\
 \vdots\qquad\qquad\quad \vdots \qquad \qquad\quad\vdots\\
 \Delta u_m + \partial_m F(x, u_1,...,u_m) + \lambda u_m = 0.
 \end{array}\right.\]
 with  $\lambda$ being the Lagrange multiplier associated to the mass constraint and $\partial_i:=\partial_{u_i}$. In particular, when $\partial_i F(x,u_1,...,u_m) = \partial_i F(x,|u_1|,...,|u_m|)$, solutions of $(\mathcal I)$ can also be viewed as
 standing waves of the following non-linear Schr\"odinger system
 $$\left\{ \begin{array}{l}
 i\partial_t \phi_1(t,x) +
 \Delta_{xx}\phi_1+ \partial_1 F(x, |\phi_1|,..,|\phi_m|) \,\phi_1 = 0,\\
 \vdots\qquad\qquad\vdots\qquad\qquad\vdots\\
 i\partial_t \phi_m(t,x) + 
 \Delta_{xx} \phi_m+\partial_mF(x,|\phi_1|,...,|\phi_m|)\,\phi_m  = 0,\\\\
 \phi_i(0,x) = \phi^0_i(x) \quad 1 \leq i \leq m.
 \end{array}\right.$$
 To our knowledge, the literature is completly silent about $(\mathcal I)$
 when $m \geq 2$ and the non-linearity $F$ does not satisfy the
 standard convexity, compactness, symmetry or monotonicity
 properties. Such a problem appears in many areas, rational mechanics and engineering for instance and especially in non-linear optics, \cite{refop1,refop2}.
\vskip6pt
In this contribution, our purpose is to prove the existence of minimizers to the problem $(\mathcal I)$ for a given function $F : \mathbb{R}^N
 \times \mathbb{R}^m \rightarrow
 \mathbb{R}$ satisfying   $F \in \mathcal D( \mathbb{R}^N \times \mathbb{R}^m)$ and 
\begin{itemize}
\item[$\mathcal A_0:$] For all $x \in \mathbb{R}^N, \vec{s}
 \in \mathbb{R}^m$, there exist $ \; A,B >0$ and $0<\ell<\frac4N$ such that for all $1\leq i\leq m$, the function $F$ satisfies
 \[0 \leq F(x, \vec{s}) \leq A(|\vec{s}|^2 + |\vec{s}|^{\ell+2})\:\:{\rm and}\:\: \partial_i F(x,\vec{s}) \leq B(|s| +
 |\vec{s}|^{\ell+1}).\]
\item[$\mathcal A_1:$] There exist $\Delta > 0,\; S > 0,\; R > 0, \alpha_1,...,\alpha_m > 0, t \in
 [0,2)$ such that for all $|x|\geq R$ and $|\vec{s}| < S$ with  $ t< N\left(1 -\frac{\alpha}{2}\right)+ 2  $ and $ \alpha =
 \displaystyle{\sum^m_{i=1}\alpha_i}$, it holds
 \[F(x, \vec{s}) > \Delta |x|^{-t}
 |s_1|^{\alpha_1}...|s_m|^{\alpha_m}.\]
 \item[$\mathcal A_2:$] For all $x \in \mathbb{R}^N , \vec{s} \in \mathbb{R}^m$ and $ \theta \geq1$, we have 
 \[F(x, \theta \vec{s})\geq
 \theta^2 F(x,\vec{s}).\]
\end{itemize}
 Moreover, we assume that there exists a periodic function $F^\infty(x,\vec s)$, that is there exists $z \in
 \mathbb{Z}^N$ such that $F^\infty(x+z, \vec{s}) = F^\infty(x,\vec{s})$ for all $\vec s\in \mathbb R^N$ and $\vec s\in \mathbb R^m$, satisfying $\mathcal{A}_1$ and 
 \begin{itemize}
 \item[$\mathcal A_3:$] There exists $0<\alpha<\frac4N$ such that it holds uniformly for all $\vec s \in \mathbb R^m$
 \[\lim_{|x|\rightarrow +\infty}\;
 \frac{F(x,\vec{s})-F^\infty(x,\vec{s})}{|\vec{s}|^2+|\vec{s}|^{\alpha+2}}=0.\]
  \item[$\mathcal A_4:$] For all $x \in \mathbb{R}^N, \vec{s}
 \in \mathbb{R}^m$, there exist $ \; A',B' >0$ and $0<\beta<\ell<\frac4N$ such that for all $1\leq i\leq m$, the function $F^\infty$ satisfies
  \[0 \leq F^\infty(x, \vec{s}) \leq A'(|\vec{s}|^{\beta+2} + |\vec{s}|^{\ell+2})\:\:{\rm and}\:\: \partial_i F^\infty(x,\vec{s}) \leq B'(|s|^{\beta+1} +
 |\vec{s}|^{\ell+1}).\]
  \item[$\mathcal A_5:$] There exists $\sigma\in\left.\left[0,\frac4N\right)\right.$ such that for all $x \in \mathbb{R}^N, \vec{s} \in \mathbb{R}^m$ and $\theta\geq 1$, it holds
  \[F^\infty(x, \theta \vec{s})
 \geq \theta^{\sigma+2}
 F^\infty (x, \vec{s}).\]
  \item[$\mathcal A_6:$] For all $x \in
 \mathbb{R}^N$ and $ s \in \mathbb{R}^m$, we have $F^\infty(x,\vec{s}) \leq F(x,\vec{s})$ with strict inequality in
 a measurable set having a positive Lebesgue measure.
 \end{itemize}
\vskip6pt
The class of nonlinearities satisfying $\mathcal A_0-\mathcal A_6$ is certainly not empty. Actually, it contains physical cases. For the sake of simplicity we shall here state the following example in the setting $m=2$ which can be extended to $m>2$. Let $k\in\mathbb N^\star$ and for all $1\leq i\leq k$, let the reals $l_{1,j},l_{2,j}>0$ such that $l_{1,j}+l_{2,j}<\frac4N$ the function 
\[F(r,\vec s)=p(r)\,|\vec s|^2 +q(r)\,\sum_{i=1}^k\,|s_1|^{l_{1,j}+1}\,|s_2|^{l_{2,j}+1},\]
 where $p,q:[0,+\infty)\mapsto \mathbb R_+$ being two bounded mapping satisfying $ p(r) \stcr 0$ and $p(r) \stcr q_\infty$ with $q_\infty\leq q(r)$ for almost all $r$ and $q_\infty<q(r)$ in a set with measure greater than $0$.
 \vskip6pt
 \noindent To our knowledge, all existing results addressed the nonlinearity of the type $F(r,\vec s)=\frac{1}{2p}\,s_1^{2p}+\frac{1}{2p}\,s_2^{2p} +\frac\beta p\,s_1^p\,s_2^p$. It is known that single mode optical fibers are not unimodal but bimodal due to the presence of birefringence which heavily influences the way of propagation along the fiber. $F$ is related to the index of refraction of the media in which the wave propagates. By Snell's law, it is not reasonable to assume that $F$ has such a form although in some situations, it provides with a good approximation of the index of refraction. We refer the reader Refs. \cite{ref1,ref2,ref3,ref4,ref5,ref6} for more detail concerning applications. Let us also mention that the example we gave above describes also the Kerr-like photorefractive media in optics. It appears in the binary mixture of Bose-Einstein condensates in two different hyperfine states.
 \vskip6pt
 \noindent Our main result is the following 
 \begin{thm}\label{thm1}
Let $\mathcal A_1-\mathcal A_6$ hold true, then there exists $\vec u_c\in\mathcal S_c$ such that $\mathcal J(\vec u_c)=\mathcal I_c$.
\end{thm}
\noindent Also, we have the following intermediate result
\begin{thm}\label{thm2}
If $\mathcal A_1$ holds true for $F^\infty$, $\mathcal A_4$ and
 $\mathcal A_5$ are satisfied, then there exists $\vec{u} \in \mathcal S_c$ such
 that $\mathcal J^\infty(\vec{u}_c) = \mathcal I^\infty_{c}$ where
 $$\mathcal J^\infty(\vec{u}) = \s\left(\frac12\;|\nabla\vec{u}|^2 - 
 F^\infty(x,\vec{u})\right) dx,$$
 $$\mathcal I^\infty_{c} = \inf \{J^\infty(\vec{u}),\qquad
 \quad \vec{u} \in \mathcal S_c\}.\eqno{(\mathcal I_\infty)}$$
\end{thm}
\noindent  Our proofs of Theorems \ref{thm1} and \ref{thm2} are based on the breakthrough concentration-compactness principle, \cite{PL1,PL2}. Such a principle states in the one-constrained setting for 
 $$(\imath)\qquad \imath_c = \inf\{j(u),\qquad \s u^2 = c^2\},$$
 where $j(u) = \displaystyle{\frac{1}{2}\s|\nabla u|^2\,dx - \s} f(x,
 u(x))\,dx$, that if $\{u_n\}_{n\in\mathbb N}$ is a minimizing sequence of the problem
 $(\imath)$, then only one of the three following scenarios can occur.
 \begin{itemize}
 \item {\it Vanishing}: $\displaystyle{\lim_{n\rightarrow +\infty}
 \sup_{y \in \mathbb{R}^N}
 \int_{B(y,R)}}u^2_n(x)dx = 0$.
 \item {\it Dichotomy}: There exists $a \in (0,c)$ such that $\forall\;
 \varepsilon > 0, \exists\; n_0 \in \mathbb{\mathbb{N}}$
 and two bounded sequences in $\1, \{u_{n,1}\}_{n\in\mathbb N}$ and $\{u_{n,2}\}_{n\in\mathbb N}$
 (all depending on $\varepsilon)$ such that for every $n \geq
 n_0$, it holds 
 $$|\s u^2_{n,1}\,dx -a^2| < \varepsilon ;\quad |\s u^2_{n,2} \,dx
 - (c^2-a^2)|< \varepsilon$$
 with $\lim_{n\rightarrow +\infty} \,dist \,supp \,(u_{n,1}, u_{n,2}) = +\infty$.
 \item {\it Compactness}: There exists a sequence $\{y_n\}_{n\in\mathbb N} \subset
 \mathbb{R}^N$ such that, for all $\varepsilon > 0$, there exists
 $R(\varepsilon)$ such that for all $n\in\mathbb N$
 $$\int_{B(y_n,R(\varepsilon))} u^2_n(x)dx \geq c^2 -
  \varepsilon.$$
 \end{itemize}
The seminal work of P.L. Lions states a general line of attack to
exclude the two first alternatives. When one knows that compactness
is the only possible case, $(\imath)$ becomes much more easier to handle. Indeed, to rule out vanishing the main ingredient is to get a strict sign of the value of $\imath_c$ (let us say $\imath_c < 0$ without loss of
generality). This can be obtained by dilatation arguments or test functions techniques .
The more delicate point is to prove that dichotomy cannot occur. For that purpose, Lions suggested a heuristic approach based on the strict subadditivity inequality 
\begin{equation}
\imath_c < \imath_a + \imath^\infty_{c-a}\quad \forall\; a \in (0,c),\label{eq1.4}\end{equation}
where 
\[\imath^\infty_c = \inf\{j^\infty(u),\qquad  \s |u|^2dx=c^2\},\] \[ j^\infty(u) = \displaystyle{\frac{1}{2}\s|\nabla u|^2\,dx - \s} f^\infty(x,
 u(x))\,dx,\] and $f^\infty$ is defined as in $\mathcal A_3$. On the other hand, we should establish suitable assumptions on $f$
 for which $j(u_n) \geq j(u_{n,1}) + j^\infty(u_{n,2}) - g(\delta)$ where
 $g(\delta) \rightarrow 0$ as $\delta \rightarrow 0$. This fact requires a deep study of the functionals $j$ and
 $j^\infty$. The continuity of $\imath_c$ and $\imath^\infty_c$ also plays a
 crucial role to show that dichotomy cannot occur. These issues do
 not seem to be discussed in the seminal paper of Lions.\vskip6pt
 When one knows that compactness is the only plausible alternative,
 the strict inequality $\imath_c < \imath^\infty_c$ is very helpful to
 prove that $(\imath)$ admits a solution. Let us mention that \eqref{eq1.4} and $\imath_c < \imath^\infty_c$ seem to be inescapable to rule out the dichotomy in Lions method. In the most interesting cases $(\imath^\infty_c \neq 0)$. In order to get $\imath_c < \imath^\infty_c$, we need first to apply the
 concentration-compactness method to the problem at infinity. This
 problem is less complicated than the original one since it has
 translation invariance properties. The key tool to prove that $\imath^\infty_c$ is achieved, that is 
 \begin{equation}\exists\; u_\infty \in S_c\quad \mbox{ such that }
 \quad j^\infty(u_\infty) = \imath^\infty_c,
 \label{eq1.6}\end{equation}
 is the strict subadditivity inequality $\imath^\infty_c < \imath^\infty_a + \imath^\infty_{c-a}$.
 On the other hand, it is quite easy to establish assumptions on
 $f$ such that for all $u\in\1$, we get $j(u) < j^\infty(u)$.Therefore, $\imath_c \leq \imath^\infty_c$. Thus, with \eqref{eq1.6}, we get $\imath_c < \imath^\infty_c$. Hence to obtain \eqref{eq1.4}, it suffices to prove that $\imath_c \leq \imath_a + \imath_{c-a}$
 which can be immediately derived from the following property
 $$f(x, \theta s) \geq \theta^2f(x,s)\quad \forall\; s \in \mathbb{R}_+,
 x \in \mathbb{R}
 \mbox{ and } \theta > 1.$$
 To study the multi-constrained variational problem $(\mathcal I)$, we will
 follow the same line of attack described in details above. Let us
 first emphasize that even for $m = 1$, it does not seems to us that
 the discussion presented in Ref. \cite{PL1,PL2} contains all the details and
 some steps are only stated heuristically. Also, to our knowledge, there are no previous results dealing with
  $(\mathcal I)$ when $m \geq 2$ and the non-linearity $F$ does not satisfy
  the classical convexity, compactness and monotonicity properties.
  Quite recently, in Ref. \cite{H}, the author was able to generalize and
  extend previous results addressed to $(\mathcal I)$ when $F$ is radial and
  supermodular (i.e $\partial_i\partial_j F \geq 0$ $\forall\; 1 \leq i \neq j \leq
  m$ when $F$ is smooth).
\vskip6pt
 In the vectorial context, the equivalent of \eqref{eq1.4} is 
  \begin{equation}\mathcal I_{c} < \mathcal I_{a} + \mathcal I_{c-a}^\infty\quad
  \forall\; 0 < a < c.\quad \label{eq1.10}\end{equation}
  We will first prove that $\mathcal I_{c} < 0$ in Lemma \ref{Lemma 3.2}. This property together with $\mathcal A_2$ will permit us to infer
  \begin{equation}\mathcal I_{c} \leq \mathcal I_{a} + \mathcal I_{c-a}\quad
  \forall\; 0 <a < c.\quad 
  \label{eq1.11}\end{equation}
  Following the same approach detailed for the scalar case, we will
  then study $(\mathcal I_\infty)$ and prove that this variational problem
  has a minimum. That is, there exists $\vec{u}^\infty_c \in \mathcal S_c$ such that
  \begin{equation}\mathcal J^\infty(\vec{u}^\infty_c) = \mathcal I^\infty_{c}.
  \label{eq1.12}\end{equation}
  This equality is obtained thanks to the subadditivity condition
  \begin{equation*}\mathcal I^\infty_{c} < \mathcal I^\infty_{a} \; +
  \mathcal I^\infty_{c-a}\quad 
  \forall\; 0 < a < c, \end{equation*}
  which is proved in part b) of Lemma \ref{Lemma 3.3}. On the other hand, $\mathcal A_6$ tells us that for all $\vec u\in \1$, we have 
  \begin{equation*}\mathcal J(\vec{u}) < \mathcal J^\infty(\vec{u}).\end{equation*}
  Therefore, with \eqref{eq1.12} we get  
  \begin{equation}\mathcal I_{c} < \mathcal I^\infty_{c}.\label{eq1.15}\end{equation}
  Now, \eqref{eq1.11} and \eqref{eq1.15} lead to \eqref{eq1.10}. Then using the properties of the splitting sequences $\vec{v}_n$
  and $\vec{w}_n$ (see appendix) and those of the functionals $ \mathcal J$ and
  $\mathcal J^\infty$ (Lemma \ref{Lemma 3.1}), we prove that any minimizing sequence of
  $(\mathcal I)$ is such that
  $$\mathcal J(\vec{u}_n) \geq \mathcal J(\vec{v}_n) + \mathcal J^\infty(\vec{w}_n) -
   \delta\quad \delta \rightarrow 0,$$
  or
  $$ \mathcal J(\vec{u}_n) \geq \mathcal J^\infty(\vec{v}_n) + \mathcal J(\vec{w}_n) - \delta.$$
  This leads to a contradiction with \eqref{eq1.10}. Therefore compactness
  occurs and we can conclude that Theorem \ref{thm1} holds true using
  \eqref{eq1.15}.
  \vskip6pt
  \noindent For the convenience of the reader, we summarize our approach
  (inspired by Lions principle) into the following steps
  \begin{itemize}
 \item[{\it i})] Obtain useful properties about the functionals $\mathcal J$ and
  $\mathcal J^\infty$ (Lemma \ref{Lemma 3.1}).
  \item[{\it ii})] Prove that $\mathcal I_{c} < 0$ and $\mathcal I^{\infty}_{c} <
  0$ (Lemma \ref{Lemma 3.2}).
 \item[{\it iii})] Show that $\mathcal I_{c} \leq \mathcal I_{a} +
  \mathcal I_{c-a}$ (Lemma \ref{Lemma 3.3}).
 \item[{\it iv})] Prove that $(\mathcal I_\infty)$ is achieved thanks to the
  strict inequality
  $$\mathcal I^\infty_{c} < \mathcal I^\infty_{a} +\mathcal  I^\infty_{c-a}.$$
  \item[{\it v})] Prove that $I_{c} < I^\infty_{c}$ (Lemma \ref{Lemma
  3.4}).
  \item[{\it vi})] The inequality $\mathcal I_{c} < \mathcal I_{a} +
  \mathcal I^\infty_{c-a}$ follows from Step {\it iii} and Step {\it v}.
  \item[{\it vii})] Only compactness can occur. In fact Step {\it ii} permits us
  to rule out vanishing. Step {\it i} and step {\it vi} will be crucial to
  eliminate dichotomy.
\end{itemize}
From now on, $|\vec s|$ will denote the modulus of the vector $ \vec{s} = (s_1,...,s_m)$ where $s_i \in\mathbb R$ for all $1\leq i\leq m$ and $m\in \mathbb N^\star$. The critical sobolev exponent will be denoted $2^\star=\frac{2N}{N-2}$. Also, if $\vec{u} = (u_1,...,u_m) \in \times ^m L^p(\mathbb R^N)  :=\vec{L}^p$, then $\nr\vec{u}\nr_{\vec{L}_p} := \displaystyle{\sum_{i=1}^m}\nr u_i\nr_p$ and equivalently for all functional spaces. The notation $\nr\cdot\nr_p$ stands for the $L^p(\mathbb R^N)$ norm and we shall write $L^p$ instead of $L^p(\mathbb R^N)$, $H^1$ instead of $H^1(\mathbb R^N)$ etc. Moreover, we shall use implicitly the obvious estimate $ \nr \vec u\nr_{\vec L_p} \leq c_p \,\sum_{i=1}^m \nr u_i\nr_{p}$ and $c_p$ the associated universal constant.
\section{A few technical Lemmata}
We start by collecting some useful Lemmas. First of all, we claim 
\begin{lem}\label{Lemma 3.1}
Let $F$ satisfies $\mathcal A_0$. Then
\begin{itemize}
\item[{\it i})] 
\begin{itemize}
\item[a)] $\mathcal J \in C^1(\vec{H}^1,\mathbb{R})$ and there exists a
constant $E > 0$ such that for all $\vec{u} \in \vec{H}^1$, it holds
$$\nr \mathcal J'(\vec{u})\nr_{\vec{H}^{-1}} \leq E \left(
\nr\vec{u}\nr_{\vec{H}^1} + \nr\vec{u}\nr^{1+
\frac{4}{N}}_{\vec{H}^1}\right).$$ 
\item[b)]  $\mathcal J^\infty \in C^1(\vec{H}^1, \mathbb{R})$ and there exists
a constant $E_\infty > 0$ such that for all $\vec{u} \in \vec{H}^1$, it holds
$$\nr \mathcal J^{\infty'}(\vec{u})\nr_{\vec{H}^{-1}} \leq E_\infty\left(\nr\vec{u}\nr_{\vec{H}^1}
+\nr\vec{u}\nr_{\vec{H}^1}^{1+ \frac{4}{N}}\right).$$
\end{itemize}
\item[{\it ii})] There exist constants $A_i, B_i > 0$ such that for all $\vec{u} \in S_c$, we have (with $\sigma,\sigma_1$ and $q,q_1$ defined in the proof below)
 $$\mathcal J(\vec{u}) \geq A_1\nr\nabla \vec{u}\nr^2_2 - A_2
 c^2 - A_3 c^{(1-\sigma)(\ell+2)q},$$
 $$\mathcal J^\infty(\vec{u}) \geq B_1\nr\nabla \vec{u}\nr^2_2
- B_2 c^{(1-\sigma_1)(\beta+2)q_1} - B_3c^{(1-\sigma)(\ell+2)q}.$$
\item[{\it iii})] \begin{itemize}
\item[a)] $\mathcal I_{c} > - \infty$ and any minimzing sequence of
$(\mathcal I)$ is bounded in $\vec{H}^1$.
\item[b)] $\mathcal I^\infty_{c} > - \infty$ and any minimizing
sequence of $(\mathcal I_\infty)$ is bounded in $\vec{H}^1$.
\end{itemize}
\item[{\it iv})] \begin{itemize}
\item[a)] The mapping $c \mapsto \mathcal I_{c}$ is continuous on $(0, +\infty)$.
\item[b)] The mapping $c \mapsto \mathcal I^\infty_{c}$ is continuous on $(0, +\infty)$.
\end{itemize}
\end{itemize}
\end{lem}
\begin{proof}
We prove the first assertion. For that purpose, we introduce a cutoff function $\varphi : \mathbb{R}^m \rightarrow \mathbb{R}$ such that $\varphi(\vec{s}) = 1 $ if $ |\vec{s}| \leq 1$, $\varphi(\vec{s})= -|\vec{s}| + 2$ if $1 \leq |\vec{s}|$ and $\varphi(\vec{s})=0$ otherwise. Now, for all $1\leq i\leq m$, we introduce 
\begin{align*}
&\partial^1_i F(x,\vec{s}) = \varphi(\vec{s}) \partial_i F(x,
\vec{s}),
&|\partial^1_iF(x,\vec{s})| \leq B(1+2^{\ell+1})|\vec{s}|,\\
&\partial^2_i F(x,\vec{s}) = (1-\varphi(\vec{s}))\partial_i
F(x,\vec{s}),
& |\partial^2_i F(x,\vec{s})|\leq 2B |\vec{s}|^{1 +\frac{4}{N}}.
\end{align*}
Also, we let $p=\frac{2N}{N+2}$ if $N\geq 3$, $p=\frac43$ if $N\leq 2$ and $q=p\,\left(1+\frac4N\right)$. The definitions above imply that $\partial^1_i F(x,.) \in
C(\vec{L}^2,L^2)$, $\partial^2_i F(x,.) \in C(\vec{L}^q,L^p)$ and
there exists a constant $K > 0$ such that 
\begin{align*}&\nr\partial^1_i F(x,\vec{u})\nr_2 \leq K\,\nr\vec{u}\nr_2,
&\quad \forall \vec{u} \in \vec{L}^2, \\&\nr\partial^2_i F(x,\vec{u})\nr_p \leq K\,\nr\vec{u}\nr_q^{1
+ \frac{4}{N}}, &\quad
\forall\; \vec{u} \in \vec{L}^q.\end{align*}
Noticing that $\vec{H}^1$ is continuously embedded in $\vec{L}^q$ since $q \in [2,
2^\star]$ for $N \geq 3$ and $q \in [2,+\infty)$ for $N \leq
2$, and $\vec{L}^p$ is continuously embedded in $\vec{H}^{-1}$ since
$p' \in [2, 2^\star]$ for $N \geq 3$ and $p' \in [2,+\infty)$
for $N \leq 2$. We can assert that $\partial_i
F(x,.)+\partial^2_iF(x,.) \in C(\vec{H}^1,\vec{H}^{-1})$ and there
exists a constant $C > 0$ such that for all $\vec{u} \in \vec{H}^1$, it holds
\[\nr\partial_i F(x,\vec{u})\nr_{\vec{H}^{-1}} \leq C\left\{\nr\vec{u}\nr_{\vec{H}^1}
 + \nr\vec{u}\nr^{1+ \frac{4}{N}}_{\vec{H}^1}\right\}.\]
On the other hand
$$\s F(x,\vec{u})\,dx \leq A(\nr\vec{u}\nr^2_2 + \nr\vec{u}\nr^{\ell+2}_{\ell+2}) \leq
C\left(\nr\vec{u}\nr^2_{\vec{H}^1} + \nr\vec{u}\nr^{\ell+2}_{\vec{H}^1}\right)$$ which
implies that $\mathcal J \in C^1(\vec{H}^1,\mathbb{R})$ by standard arguments
of differential calculus . Thus 
$$\mathcal J'(\vec{u}) \vec{v} = \s\; \left(\sum^m_{i=1} \nabla u_i\;
 \nabla v_i - \partial_i F(x,\vec{u})v_i\right)\,dx,\quad \forall\; \vec{u},
 \vec{v} \in \vec{H}^1.$$
Therefore,
 $$\nr\mathcal J'(\vec{u})\nr_{\vec{H}^{-1}} \leq C\left\{\nr\vec{u}\nr_{\vec{H}^1} + \nr\vec{u}\nr^{1 +
 \frac{4}{N}}_{\vec{H}^1}\right\}, \quad \forall\; \vec{u} \in \vec{H}^1.$$
The assertion concerning the functional $\mathcal J^\infty$ can be proved similarly and we skip the proof for the sake of shortness. Now, we turn to the proof of the second assertion. Let $\vec{u}:=(u_1,...,u_m) \in S_c$. Using $\mathcal A_0$, we have 
\[\s F(x,\vec{u})\,dx \leq A c^2 +
A\displaystyle{\sum^m_{i=1}\s |u_i(x)|^{\ell+2}\,dx}.\] 
Now, let  $\sigma ={\frac{N}{2} \frac{\ell}{\ell+2}}$, then for all  $1 \leq i \leq m$, thanks to the Gagliardo-Nirenberg inequality, we have
\begin{equation}\nr u_i\nr^{\ell+2}_{\ell+2} \leq A" \nr u_i\nr_2^{(1-\sigma)(\ell+2)}
\nr\nabla u_i\nr^{\sigma(\ell+2)}_2. \label{eq3.4}
\end{equation}
Next, letting $\varepsilon > 0,\;p = \frac{4}{N\ell}$ and $q$ such that
$\displaystyle{\frac{1}{p} + \frac{1}{q}} = 1$, then Young's inequality leads to 
$$\nr u_i\nr^{\ell+2}_{\ell+2} \leq \left\{ \frac{A"}{\varepsilon}
\nr u_i\nr_2^{(1-\sigma)(\ell+2)}\right\}^q\frac{1}{q} +
\frac{N\ell}{4}\{\varepsilon^{\frac{4}{N\ell}}\nr\nabla u_i\nr^2_2\}.$$
Consequently, 
\[\mathcal J(\vec{u}) \geq \left\{\frac{1}{2} -\frac{AN\ell}{4}
\varepsilon^{\frac{4}{N\ell}} \right\} \nr\nabla \vec{u}\nr^2_2 - A^2
c^2- \frac{AA"^q}{q\varepsilon^q} m
c^{(1-\sigma)(\ell+2)q}.\] 
Taking $\varepsilon$ such that
$\displaystyle{\frac{1}{2} -
\frac{AN\ell}{4}}\varepsilon^{\frac{4}{N\ell}} \geq 0$, we prove
that $\mathcal J$ is bounded from below in $\vec{H}^1$. To show that all
minimizing sequences of $(\mathcal I)$ are bounded in $\vec{H}^1$, it suffices
to take the latter inequality with a strict sign.
\begin{rem}
On the one hand, if we allow $\ell = \frac{4}{N}$ in $\mathcal A_0$, the minimization problem $(\mathcal I)$
 makes sense for sufficiently small values of $c$ since in \eqref{eq3.4}, we then have
$\sigma = \frac{2}{\ell+2}$ and $(1-\sigma)(\ell+2) = \frac{4}{N}$.
Therefore
\begin{eqnarray*}
\nr u_i\nr^{\ell+2}_{\ell+2} \leq A" c^{(1-\sigma)(\ell+2)}\nr\nabla u_i\nr^2_2\\
\leq A" c^{4/N} \nr\nabla u_i\nr^2_2.
\end{eqnarray*}
$$\mathcal J(\vec{u}) \geq \left\{ \frac{1}{2} - AA" c^{4/N}\right\} \nr\nabla
\vec{u}\nr^2_2 - Ac^2.$$
Thus, if $c < (\frac{1}{2AA"})^{N/4}$, the minimization problem $(\mathcal I)$
is still well-posed. On the other hand, if $\ell > \frac{4}{N}$, we can prove that $\mathcal I_{c} = -
\infty$.
\end{rem}
Next, under a slight modifications of the argument we used above, we can easily obtain for $\mathcal J^\infty$ the following estimate
\begin{eqnarray*}
\mathcal J^\infty (\vec{u}) &\geq& \Big\{\frac{1}{2} - A^{(3)}
\varepsilon^{\frac{4}{N\ell}}\Big\} \nr\nabla \vec{u}\nr^2_2 -
\frac{A^{(4)}m}{q_1\varepsilon^{q_1}} c^{(1-\sigma_1)(\beta+2)q_1}\\
&-& \frac{A^{(5)}mc^{(1-\sigma)(\ell+2)q}}{q\varepsilon^q},
\end{eqnarray*}
with $\sigma= \displaystyle{\frac{N}{2} \frac{\beta}{\beta+2}}$
and $\sigma_1 = \displaystyle{\frac{N}{2} \frac{\ell}{\ell +2}}$ and
$q_1$ is also defined as in the previous proof. The assertion $iii)$ is a straightforward consequence of the estimates of the second point. Therefore, we are kept with the proof of the last point. Consider $c > 0$ and a sequence
$\{c^n\}_{n\in\mathbb N}$ such that $c^n \rightarrow c$. For any $n$, there exist $u_{n} \in
\mathcal S_{c^n}$ such that
$$\mathcal I_{c^n} \leq \mathcal J(u_{n,1},...,u_{n,m})
 \leq \mathcal I_{c^n} + \frac{1}{n}.$$
Thanks to the first estimate of the assertion $ii)$, we can easily see that there exists a constant $K >
0$ such that $\nr\vec{u}_n\nr_{\vec{H}^1} \leq K$ for all $ n \in
\mathbb{N}$. Now, we introduce $\vec{w}_n =
(w_{n,1},...,w_{n,m})$ where $w_{n} = \frac{c}{c^n} u_{n}$. Then, we have obviously $\vec{w}_n \in \mathcal S_c$ and 
\begin{eqnarray*}
\nr\vec{u}_n - \vec{w}_n\nr_{\vec{H}^1} &\leq & \left|\frac{c}{c^n}-1\right|\nr\vec(u)_{n}\nr_{H^1}.
\end{eqnarray*}
In particular, there exists $n_1$ such that
$$\nr\vec{u}_n - \vec{w}_n\nr_{\vec{H}^1} \leq K +1\quad \mbox{ for all } n \geq n_1.$$
Now, it follows from the first assertion that 
\[\nr \mathcal J'(\vec{u})\nr_{\vec{H}^{-1}} \leq L(K)\quad \mbox{ for } \nr\vec{u}\nr_{\vec{H}^1}
\leq 2K+1.\]
Therefore for all $n \geq n_1$, we have 
\begin{eqnarray*}
|\mathcal J(\vec{w}_n)-\mathcal J(\vec{u}_n)| &=& |\int^1_0 \frac{d}{dt} \mathcal J(t\vec{w}_n
+
(1-t)\vec{u}_n)dt|,\\
&\leq& \sup_{\nr\vec{u}\nr_{\vec{H}^1} \leq 2K+1}
|J'(\vec{u})|_{\vec{H}^{-1}}\nr\vec{u}_n - \vec{w}_n\nr_{\vec{H}^1},\\
&\leq& L(K) K  \left|1 - \frac{c}{c^n}\right|.
\end{eqnarray*}
Eventually, we have
\[\mathcal I_{c} \geq \mathcal J(\vec{u}_n) - \frac{1}{n} \geq \mathcal J(\vec{w}_n) +KL(K)
 \left|1-\frac{c}{c^n}\right| - \frac{1}{n}.\]
 Thus $\liminf_{n\rightarrow +\infty} \mathcal I_{c^n} \geq \mathcal I_{c}$. On the other hand, there exists a sequence $\vec{u}_n \in \mathcal S_c$ such that $\mathcal J(\vec{u}_n)\stcn \mathcal I_{c}$ and, thanks to the first assertion, there exists $K > 0$ such that $\nr\vec{u}_n\nr_{\vec{H}^1} \leq
 K$. Now, we set $w_{n} = \displaystyle{\frac{c_n}{c}}u_{n}$. following the argument above, we have
 $\vec{w}_n =
 (w_{n,1,...,}w_{n,m})
 \in \mathcal S_{c_n}$, $c_n = (c^1_n,...,c^m_n)$
 and
 $$\nr\vec{u}_n - \vec{w}_n\nr_{\vec{H}^1}
 \leq K\left|1 - \frac{c_n}{c}\right| \nr\vec(u)_{n}\nr_{H^1}.$$
 Once again, as done previously, we get 
 $$|\mathcal J(\vec{w}_n) -\mathcal J(\vec{u}_n)| \leq K L(K) \left|1 - \frac{c_n}{c}\right|,$$
 which implies that
 $$\mathcal I_{c} \leq \mathcal J(\vec{w}_n) \leq \mathcal J(\vec{u}_n)+L(K)K 
 \left|1 - \frac{c_n}{c}\right|.$$
 Thus $\limsup_{n\rightarrow +\infty} \mathcal I_{c^n} \leq \mathcal I_{c}$ and we
 conclude. The equivalent assertion for $\mathcal I_\infty$ follows using the same argument.
\end{proof}
\noindent We shall need the following second technical Lemma
\begin{lem}\label{Lemma 3.2}
Let $F$ such that $\mathcal A_0$ and $\mathcal A_1$ hold, then $\mathcal I_c<0$.
\end{lem}
\begin{proof}
Let $\varphi$ be a radial and radially decreasing
 function such that $\nr\varphi\nr_2 = 1$ and we  set $\varphi_i = c_i
 \varphi$. Also, let $0 <\lambda \ll 1$ and $\vec{\Phi}_\lambda(x) = \lambda^{N/2}
 \vec{\Phi}(\lambda x) :=
 \lambda^{N/2} (\varphi_1(\lambda x)),...,\varphi_m(\lambda x))$. Then, we have 
 \begin{eqnarray*}
 \mathcal J(\vec{\Phi}_\lambda) &=& \lambda^2 \nr\nabla \vec{\Phi}\nr^2_2 - \s
 F(x, \lambda^{N/2} \varphi_1(\lambda x),...,\lambda^{N/2}
 \varphi_m(\lambda x))dx,\\
 &\leq& \lambda^2 \nr\nabla \vec{\Phi}\nr^2_2 - \int_{|x|\geq R}\,F(x, \lambda^{N/2} \varphi_1(\lambda x),...,\lambda^{N/2}
 \varphi_m(\lambda x))dx,\\
 &\leq&\lambda^2 \nr\nabla\vec{\Phi}\nr^2_2 - \lambda^{\frac{N}{2}\alpha}
 \Delta \int_{|x|\geq R}|x|^{-t_i} \varphi_1^{\alpha_1}(\lambda
 x)...\varphi_m^{\alpha_m}(\lambda x)dx.
 \end{eqnarray*}
Applying the change of variable $y = \lambda x$ leads to 
that 
$$\mathcal J(\vec{\Phi}_\lambda) \leq \lambda^2 \nr\nabla \vec{\Phi}\nr^2_2
 - \lambda^{\frac{N}{2}\alpha} \lambda^{-N}\Delta
 \lambda^{t_i}\int_{|y|
 \geq \lambda R}
 \varphi_1^{\alpha_1}(y),..., \varphi_m^{\alpha_m}(y)dy.$$
 Now, since $0 < \lambda \ll 1$, we get
 \begin{eqnarray*}
 \mathcal J(\vec{\Phi}_\lambda) &\leq& \lambda^2 \nr\nabla \vec{\Phi}\nr^2_2 -
 \lambda^{\frac{N}{2} -N+t_i} \int_{|y| \geq R} |y|^{-t_i}
 \varphi_1^{\alpha_1}(y) ...\varphi_m^{\alpha_m}(y)dy,\\
 &\leq& \lambda^2 \{C_1 - \lambda^{\frac{N}{2} \alpha -N + t_i-2}C_2\}.
 \end{eqnarray*}
 The result follows after observing that $\lambda \ll 1$ and $\frac{N}{2} \alpha -N+t_i-2 >
 0$.
 \end{proof}
\begin{rem}
The strict negativity of the infimum is also discussed in Ref. \cite{H} where the author provides other type of assumption ensuring this.
\vskip6pt
 \noindent Now, we have the following Lemma and we refer to Ref. \cite{PL2} for a proof.
\end{rem}
\begin{lem}\label{Lemma 3.3}
We have the following facts
\begin{itemize}
\item[{\it i})] If $F$ satisfies $\mathcal A_0,\mathcal A_1$ and $\mathcal A_2$, then for all $c>0$ and all $a\in(0,c)$ it holds $\mathcal I_{c} \leq \mathcal I_{a}+\mathcal I_{c-a}$.
 \item[{\it ii})] If $F$ satisfies $\mathcal A_2, \mathcal A_4$ and $\mathcal A_1$ holds true for $F^\infty$, then for all $c>0$ and all $a\in(0,c)$ it holds $\mathcal I^\infty_{c} < \mathcal I^\infty_{a} + \mathcal I^\infty_{c-a}$.
 \end{itemize}
\end{lem}
\noindent As a consequence, we have the following 
\begin{lem}\label{Lemma 3.4}
Let $F$ satisfies $\mathcal A_0, \mathcal A_1, \mathcal A_2$ and $\mathcal A_5, \mathcal A_1$ hold true for $F^\infty$, then for all $c>0$ and all $a\in(0,c)$ we have  \[I_{c} < I_{a} + I^\infty_{c-a}.\]
\end{lem}
\section{Proof of Theorem \ref{thm2}}
 This section is devoted to the proof of Theorem \ref{thm2}. For that purpose, let $\{\vec{u}_n\}_{n\in\mathbb N}$ be a minimizing sequence of the problem $(\mathcal I_\infty)$. We proceed by concentration-compactness scenario's elimination. First of all, we prove that vanishing does not occur. We proceed by contradiction and assume that vanishing holds true. Therefore, using Lemma I.1 of Ref. \cite{PL2} that
  $\nr|\vec{u}_n|\nr_p \stcn 0$ as for all $p \in (2,2^\star)$. Thanks to assumption $\mathcal A_4$, we have 
  $$\s F^\infty(x, \vec{u}_n(x))\,dx \leq
  \left\{\nr|\vec{u}_n|\nr^{\beta+2}_{\beta+2}+\nr|\vec{u}_n|
  \nr^{\ell+2}_{\ell+2} \right\}.$$
  Therefore, $ \s\;F^\infty(x, \vec{u}_n(x) \,dx \stcn 0 $, hence $\liminf_{n\rightarrow + \infty} \mathcal J^\infty(\vec{u}_n) \geq 0$,
  contradicting the fact that $I^\infty_{c} < 0$. Thus vanishing does not occur. Now, we use the notation introduced in the appendix and eliminate the dichotomy scenario. For all $n \geq n_0 $, we have  
  \begin{eqnarray*}
  \lefteqn{\mathcal J^\infty(\vec{u}_n) - \mathcal J^\infty(\vec{v}_n)-\mathcal J^\infty(\vec{w}_n)
  =\displaystyle{\frac{1}{2} \s}\left(|\nabla \vec{u}_n|^2 - |\nabla \vec{v}_n|^2\right)\,dx}\\
  &-&  \displaystyle{\s} \left(F^\infty(x,\vec{u}_n) -
  F^\infty(x,\vec{v}_n)-F^\infty(x,\vec{w}_n) \right)\,dx,\qquad\qquad\qquad\\
  &=&\displaystyle{\frac{1}{2} \s}\left(|\nabla \vec{u}_n|^2 - |\nabla \vec{v}_n|^2\right)\,dx- \displaystyle{\s} \left(F^\infty(x,\vec{u}_n) -
  F^\infty(x,\vec{v}_n+\vec{w}_n)\right)\,dx,\\
  &\geq&  - \varepsilon - \displaystyle{\s} \left(F^\infty(x,\vec{u}_n) -
  F^\infty(x,\vec{v}_n+\vec{w}_n)\right)\,dx.\\
  \end{eqnarray*}
In the estimate above, we used the fact that $\textrm{Supp}\: \vec{v}_n \cap \textrm{Supp}\: \vec{w}_n = \emptyset$. Now since the sequences $\{\vec{w}_n\}_{n\in\mathbb N}, \{\vec{v}_n\}_{n\in\mathbb N}$ and $\{\vec{w}_n\}_{n\in\mathbb N}$ are
  bounded in $\vec{H}^1$, it follows from the proof of Lemma \ref{Lemma 3.1}
  that there exist $C, K > 0$ such that 
  \begin{eqnarray*}
  |\displaystyle{\s} \left(F^\infty(x, \vec{u}_n)\right.&-&\left.F^\infty(x,\vec{v}_n +
  \vec{w}_n)\right)\,dx| \\ &\leq& \sup_{\nr\vec{u}\nr_{\vec{H}^1} \leq K}\displaystyle{\sum^m_{i=1}}
  \nr\partial_iF^\infty(x,\vec{u})\nr_{\vec{H}^{-1}}\nr\vec{u}_n-(\vec{v}_n +
  \vec{w}_n)\nr_{\vec{H}^1}, \\
 & \leq& \displaystyle{\sup_{\nr\vec{u}\nr_{\vec{H}^1}\leq K}\sum^m_{i=1}}
\nr\partial^1_i F^\infty(x,\vec{u})\nr_{\vec{L}^2}\nr\vec{u}_n -
(\vec{v}_n+\vec{w}_n)\nr_{\vec{L}^2},\\
&+& \displaystyle{\sup_{\nr\vec{u}\nr_{\vec{H}^1}\leq K} \sum^m_{i=1}}
\nr\partial^2_i F^\infty(x,\vec{u})\nr_{\vec{L}^p} \nr\vec{u}_n -
(\vec{v}_n +
\vec{w}_n)\nr_{\vec{L}^{p'}}, \\
&\leq& C \displaystyle{\sup_{\nr\vec{u}\nr_{\vec{H}^1} \leq K}}
\nr\vec{u}\nr_{\vec{L}^2} \nr\vec{u}_n-(\vec{v}_n +
\vec{w}_n)\nr_{\vec{L}^2}\\
&+&C \displaystyle{\sup_{\nr\vec{u}\nr_{\vec{H}^1}\leq K}} \nr\vec{u}\nr^{1 +
\frac{4}{N}}_{L^q} \nr\vec{u}_n - (\vec{v}_n +
\vec{w}_n)\nr_{\vec{L}^{p'}}, \\
&\leq& C_1 K\nr\vec{u}_n - (\vec{v}_n + \vec{w}_n)\nr_{\vec{L}^2} \\ &+& C_2
K^{1 + \frac{4}{N}} \nr\vec{u}_n - (\vec{v}_n +
\vec{w}_n)\nr_{\vec{L}^{p'}}.
\end{eqnarray*}
Thus, we get 
\begin{align*}
\mathcal J^\infty(\vec{v}_n) - \mathcal J^\infty(\vec{v}_n)-\mathcal J^\infty(\vec{w}_n)
  &\geq - \varepsilon -C_1 K\nr\vec{u}_n -
  (\vec{v}_n + \vec{w}_n)\nr_{\vec{L}^2} \\ &- C_2 K^{1 + \frac{4}{N}}
  \nr\vec{u}_n - (\vec{v}_n + \vec{w}_n)\nr_{\vec{L}^{p'}}.
  \end{align*}
 Given any $\delta > 0$, using the properties of the sequences $\{\vec{v}_n\}_{n\in\mathbb N}$ and
  $\{\vec{w}_n\}_{n\in\mathbb N}$, we can find $\varepsilon_\delta \in (0,\delta)$ such that 
   \[ 
   \mathcal J^\infty(\vec{u}_n) - \mathcal J^\infty(\vec{v}_n)
  -\mathcal J^\infty(\vec{w}_n) \geq -\delta.
  \] 
  Now let $a^2_{n} (\delta) = \s v^2_{n}\,dx$ and $b^2_{n} (\delta) = \s w^2_{n}\,dx$, passing to a subsequences if necessary, we may suppose that $a^2_{n}(\delta) \stcn a^2(\delta)$ and
 $b_{n}^2 (\delta)\stcn b^2(\delta)$
  where $|a^2(\delta)-a^2|\leq \varepsilon_\delta < \delta$
and $|b^2(\delta) - (c^2-a^2)| \leq \varepsilon_\delta <
\delta$. Recalling that $c\mapsto \mathcal I^\infty_{c}$ is continuous, we find that
\begin{eqnarray*}
\mathcal I^\infty_{c} \geq \lim_{n\rightarrow + \infty}
\mathcal J^\infty(\vec{u}_n)&\geq& \liminf \{\mathcal J^\infty(\vec{v}_n) +
\mathcal J^\infty(\vec{w}_n)\} - \delta,\\
&\geq& \lim\inf \{\mathcal I^\infty_{a,(\delta)} +
\mathcal I^\infty_{b,(\delta)}\} - \delta,\\
&\geq& \mathcal I^\infty_{a,(\delta)} +
\mathcal I^\infty_{b,(\delta)}\} - \delta.
\end{eqnarray*}
Eventually, letting $\delta$ goes to zero and using again the continuity of
$\mathcal I^\infty_{c}$, we get
$$\mathcal I^\infty_{c} \geq \mathcal I^\infty_{a} +
\mathcal I^\infty_{\sqrt{c^2-a^2}}.$$
This contradicts  Lemma \ref{Lemma 3.3}. Thus, dichotomy does not occur and we conclude that the compactness scenario holds true. Hence, there exists $\{y_n\}_{n\in\mathbb N} \subset
\mathbb{R}^N$ such that for all $\varepsilon > 0$ we have
$$\int_{B(y_n,R(\varepsilon))} u^2_{n} \,dx\geq c^2 - \varepsilon.$$
 For all $n \in \mathbb{N}$, we can choose $z_n \in \mathbb{Z}^N$
 such that  $y_n - z_n \in
 [0,1]^N$. Now we set $\vec{v}_n(x) = \vec{u}_n(x+z_n)$, we certainly have that
 $\nr\vec{v}_n\nr_{\vec{H}^1} = \nr\vec{u}_n\nr_{\vec{H}^1}$ is bounded. Therefore, passing to a subsequence if necessary, we may assume that $\vec{v}_n \rightharpoonup
 \vec{v}$ in $\vec{H}^1$. In particular $\vec{v}_n \rightharpoonup
 \vec{v}$ weakly in $\vec{L}^2$ and $\nr v_{n}\nr^2_2 = c^2$.
 However, we have 
 \begin{eqnarray*}
 \s\; |\vec{v}|^2\, dx \geq \int_{B(0,R(\varepsilon)
 +\sqrt{N})}|\vec{v}|^2 \,dx&=&\lim_{n\rightarrow + \infty}\int_{B(0,R(\varepsilon)+\sqrt{N})} |\vec{v}_n|^2\,dx\\
 &=& \lim_{n\rightarrow + \infty} \int_{B(z_n,R(\varepsilon)+\sqrt{N})}|\vec{v}_n|^2\,dx.
 \end{eqnarray*}
 Since $|y_n-z_n| \leq \sqrt{N}$, we have 
 $$\int_{B(z_n,R(\varepsilon)+\sqrt{N})}|\vec{u}_n|^2\,dx \geq
 \int_{B(y_n,R(\varepsilon))}
 |\vec{u}_n|^2\,dx \geq c^2 - \varepsilon.$$
 Hence, for all $\varepsilon >0$ we have  \begin{equation}
 \nr\vec{v}\nr^2_{\vec{L}^2} \geq c^2 -\varepsilon \Rightarrow \nr\vec{v}\nr^2_{\vec{L}^2} \geq c^2. \label{eq3.14}
 \end{equation}
On the other hand $\nr\vec{v}\nr_{\vec{L}^2} \leq \liminf_{n\rightarrow + \infty}\nr\vec{v}_{n}\nr_{\vec{L}^2}$, thus 
 \begin{equation}\nr\vec{v}\nr_{\vec{L}^2} \leq c. \label{eq3.15}\end{equation}
 Thus combining \eqref{eq3.14} and \eqref{eq3.15}, we get $\nr\vec{v}\nr^2_{\vec{L}^2} = c^2$, thus $\nr\vec{v}-\vec{v}_n\nr_{\vec{L}^2}
 \rightarrow 0$ as $n\rightarrow +\infty$. Furthermore, by the periodicity of $F^\infty$, we see that 
 $\mathcal J^\infty(\vec{u}_n) = \mathcal J^\infty(\vec{v}_n) \stcn
 \mathcal I^\infty_{c}$ and $\vec{v}_n \stcn\vec{v}$ in $\vec{L}^p$ for all $p \in
 [2,2^\star)$. It follows that $\vec{v}_n \stcn \vec{v}$ in $\vec{H}^1$ and
 consequently $\displaystyle{\s\; F^\infty(x,\vec{v}_n)} \,dx \stcn
 \displaystyle{\s}F^\infty(x, \vec{v})\,dx$ which implies that $\mathcal J^\infty(\vec{v})
 = \mathcal I^\infty_{c}$.\\
\section{Proof of Theorem \ref{thm1}}
In this section, we prove Theorem \ref{thm1}. Let $(\vec{u}_n)_{n\in\mathbb N}$ denotes a minimizing sequence of $(\mathcal I)$
 and we will again make use of the notation introduced in the appendix. As before, we start by showing that vanishing does not occur by proceeding by contradiction. Indeed, if it occurs, it follows from Lemma I.1 of Ref. \cite{PL2}
  that $\nr|\vec{u}_n|\nr_p \stcn
 0$ for $p \in (2,2^\star)$. Combining $\mathcal{A}_0$ and $\mathcal A_3$, we get that for all  $\delta > 0$ there exists $ R_\delta > 0$ such that for all $|x|\geq R_\delta$ we have 
 $$F(x,\vec{s}) \leq \delta(|\vec{s}|^2 + |\vec{s}|^{\alpha+2})
  + A'(|\vec{s}|^{\beta+2}+|\vec{s}|^{\ell+2}).$$
 Thus, $$\displaystyle{\int_{|x| \geq R_\delta}} F(x,\vec{u}_n)\,dx \leq
 \delta(\nr\vec{u}_n\nr^2_2 + \nr\vec{u}_n\nr^{\alpha+2}_{\alpha+2}) +
 A'(\nr\vec{u}_n\nr^{\beta+2}_{\beta+2}+
 \nr\vec{u}_n\nr^{\ell+2}_{\ell+2}).$$
 Therefore, $$\limsup_{n\rightarrow + \infty}
 \int_{|x| \geq R_\delta} F(x, \vec{u}_n)\,dx\leq \delta c^2.$$
On the other hand
\begin{eqnarray*}
\int_{|x|\leq R_\delta} F(x,\vec{u}_n) dx &\leq& A \int_{|x|\leq
R_\delta} \left(|\vec{u}_n|^2 + |\vec{u}_n|^{\ell+2}\right)\,dx,\\
&\leq& A \left(\nr\vec{u}_n\nr^{\ell+2}_{\ell+2}
|R_\delta|^{\frac{\ell}{\ell+2}} +
\nr\vec{u}_n\nr^{\ell+2}_{\ell+2}\right)\;\stcn0.
\end{eqnarray*}
Hence, for any $\delta > 0$ we see that
$$\limsup_{n\rightarrow +\infty}\s\;F(x, \vec{u}_n) \,dx< \delta c^2,$$
and so
$$\lim_{n\rightarrow +\infty} \s\,F(x,\vec{u}_n)\,dx = 0.$$
 The contradiction follows since we know that  $ \mathcal J(\vec{u}_n)\stcn \mathcal I_{c} < 0$. Now, we show that dichotomy does not occur. We argue again by contradiction and suppose first that the sequence $\{y_n\}_{n\in\mathbb N}$ is bounded. We write
\begin{eqnarray*}
J(\vec{u}_n) -J(\vec{v}_n) - J^\infty(\vec{w}_n) &=&
\frac{1}{2} \s \left( |\nabla \vec{w}_n|^2 - |\nabla \vec{v}_n|^2 -\nabla
\vec{w}_n|^2\right)\,dx \\
&-&\s\left( F(x,\vec{u}_n)-F(x,\vec{v}_n) -F(x,\vec{w}_n) \right)\,dx\\
 &+& \s\left( F^\infty(x,\vec{w}_n)-F(x,\vec{w}_n)\right)\,dx,\\
& \geq& - \varepsilon - \s\left( F (x,\vec{u}_n) - F(x,\vec{v}_n+\vec{w}_n)\right)\,dx\\
&+&  \s\left( F^\infty (x,\vec{w}_n)-F(x,\vec{w}_n)\right)\,dx,\\
&\geq& - \varepsilon - \s\left( F(x,\vec{u}_1)-F(x,\vec{v}_n + \vec{w}_n)\right)\,dx \\
&+& \int_{|x-y_n|\geq R_n} \left(F^\infty(x,\vec{w}_n)-F(x,\vec{w}_n)\right)\,dx.
\end{eqnarray*}
We used the fact that $\textrm{Supp}\:\vec{v}_n \cap \textrm{Supp} \:\vec{w}_n =
\emptyset$. Now using the same argument as before, it follows that given $\delta >
0$, we can choose $\varepsilon = \varepsilon_\delta \in (0,\delta)$
such that
$$- \varepsilon - \s\left( F(x,\vec{u}_n) - F(x, \vec{v}_n + \vec{w}_n)\right) \geq - \delta.$$
Therefore, we get 
$$\mathcal J(\vec{u}_n) -\mathcal J(\vec{v}_n) - \mathcal J^\infty(\vec{w}_n) \geq
- \delta + \int_{|x-y_n|\geq R_n} \left(F^\infty(x,\vec{w}_n)
-F(x,\vec{w}_n)\right)\,dx.$$
Given any $\eta > 0$, we can find $R > 0$ such
that for all $\vec{s}$ and $|x| \geq R$
$$|F^\infty(x,\vec{s})-F(x,\vec{s})| \leq \eta (|\vec{s}|^2 +|\vec{s}|^{\alpha+2}).$$
Now, since $R_n\stcn +\infty$ and we are supposing that
$\{y_n\}_{n\in\mathbb N}$ is bounded, we have for $n$ large enough
$$\{x : |x-y_n|\geq R_n\} \subset \{x : |x|\geq R\}.$$
From this and the boundedness of $\{\vec{w}_n\}_{n\in\mathbb N}$ in $\vec{H}^1$, it
follows that
$$\lim_{n\rightarrow + \infty}\int_{|x-y_n|\geq R_n}
 \left(F^\infty(x,\vec{w}_n)-F(x,\vec{w}_n) \right)\,dx= 0.$$
Now, let $a^2_{n} = \sum^m_{i=1} \int_{{\mathbb R}^N} v^2_{n,i}\,dx$ and $b^2_{n}  = \sum^m_{i=1} \int_{{\mathbb R}^N} w^2_{n,i}\,dx$. Passing to a subsequences if necessary, we may suppose that 
$a^2_{n}(\delta) \stcn a^2(\delta)$ and $b^2_{n}(\delta) \stcn b^2(\delta)$
where $|a^2(\delta)-a^2| \leq \varepsilon_\delta < \delta$ and
$|b^2(\delta)-(c^2-a^2)| \leq \varepsilon_\delta < \delta$.
Recalling that the mappings $c\mapsto\mathcal I_{c}$ and $c\mapsto\mathcal I^\infty_{c}$ are
continuous we find that
\begin{eqnarray*}
\mathcal I_{c} &=& \lim_{n\rightarrow + \infty} \mathcal J(\vec{u}_n) \geq
\liminf_{n\rightarrow + \infty}\{\mathcal J(\vec{v}_n)+\mathcal J^\infty(\vec{w}_n)\}
-\delta,\\
&\geq& \liminf_{n\rightarrow + \infty}
\{\mathcal I_{a_1(\delta),...,a_{n}(\delta)} +
\mathcal I_{b_1(\delta),...,b_{n}(\delta)}\}-\delta.
\end{eqnarray*}
Thus, $\mathcal I_{c} \geq \mathcal I_{a(\delta)}
 + \mathcal I_{b(\delta)} -\delta$. Letting $\delta \rightarrow 0$ we get 
 $$\mathcal I_{c} \geq \mathcal I_{a} + \mathcal I_{\sqrt{c^2-a^2}}.$$
 Therfore, the sequence $\{y_n\}_{n\in\mathbb{N}}$ cannot be bounded and passing to a
 subsequence if necessary, we may suppose that $|y_n| \rightarrow +\infty$.
 Now we obtain a contradiction with Lemma \ref{Lemma 3.3} by using similar
 arguments applied to $\mathcal J(\vec{u}_n)-\mathcal J^\infty(\vec{v}_n) -
 \mathcal J(\vec{w}_n)$ to show that $\mathcal I_{c} \geq \mathcal I_{a} +
 \mathcal I_{\sqrt{c^2-a^2}}$ and therefore prove that dichotomy cannot occur. Eventually, the compactness occurs. According to the appendix, there exists $\{y_n\}_{n\in\mathbb N} \subset \mathbb{R}^N$ such
 that for all $\varepsilon>0$
 $$\int_{B(y_n,R(\varepsilon))} \left(u^2_{n,1} +...+u^2_{n,m}\right)\,dx \geq c^2 - \varepsilon.$$
 Let us first prove that the sequence $\{y_n\}_{n\in\mathbb N}$ is bounded . By contradiction, if it is not the case, we may assume that $|y_n|\stcn +\infty$ by
 passing to a subsquence. Now we can choose $z_n \in \mathbb{Z}^N$ such that $y_n -z_n \in
 [0,1]^N$. Setting $\vec{v}_n(x) = \vec{u}_n(x+z_n)$, we can
 suppose that $\vec{v}_n \rightharpoonup \vec{v}$ weakly in $\vec{H}^1$
 and
 $$\nr\vec{v}_n-\vec{v}\nr_{\vec{L}^2} \stcn 0\quad \textrm{ for } \quad2 \leq p \leq 2^\star,$$
 $$\mathcal J^\infty(\vec{v}_n) = \mathcal J^\infty(\vec{u}_n).$$
 On the other hand, we have 
 \begin{eqnarray*}
\mathcal J(\vec{u}_n) - \mathcal J^\infty(\vec{u}_n) &=& \s \left(
  F^\infty(x,\vec{u}_n)-F(x,\vec{u}_n)\right)\,dx,\\
  &=& \s \left(F^\infty (x,\vec{v}_n)-F(x-z_n,\vec{v}_n)\right)\,dx.
 \end{eqnarray*}
 Now, given $\varepsilon > 0$ it follows from $\mathcal A_3$ that there
 exists $R > 0$ such that 
 \begin{eqnarray*}
 |\int_{|x-z_n|\geq R} \left(F^\infty(x,\vec{v}_n)\right. &-&\left.\left. F(x-z_n ,\vec{v}_n)\right)\,dx\right| \\ &=& \Big|\int_{|x-z_n|\geq R}
 \left(F^\infty(x-z_n,\vec{v}_n)-F(x-z_n,\vec{v}_n)\right)\,dx\Big|,\\
 &\leq&\varepsilon \int_{|x-z_n|\geq R} \left(|\vec{v}_n|^2 +
 |\vec{v}_n|^{\alpha+2} \right)\,dx, \\&\leq& \varepsilon
 C\left(\nr\vec{v}_n\nr^2_{\vec{H}^1} +
 \nr\vec{v}_n\nr^{\alpha+2}_{\vec{H}^1}\right) \leq  \varepsilon D,
  \end{eqnarray*}
  since $ \vec{v}_n$ is bounded in $\vec{H}^1$. Next, since $|z_n|\stcn +\infty$, there exists
  $n_R > 0$ such that for all $n \geq n_R$ we have 
  \begin{eqnarray*}
  \Big|\int_{|x-z_n|\leq R} \left(F^\infty(x,\vec{v}_n)\right.&-& \left.F(x-z_n,\vec{v}_n)\right)\,dx\Big|\\ &\leq& |\int_{|x| \geq \frac{1}{2} |z_n|}
  \left(F^\infty(x,\vec{v}_n)-F(x-z_n, \vec{v}_n)\right)\,dx,\\
  &\leq& K\int_{|x| \geq \frac{1}{2} |z_n|}\left(|\vec{v}_n|^2 +
  |\vec{v}_n|^{\ell+2}\right)\,dx,\\
  &\leq& K\int_{|x| \geq \frac{1}{2}|z_n|}\left(|\vec{v}|^2 +
  |\vec{v}|^{\ell+2}\right)\,dx \\&+& K\int_{|x|\geq \frac{1}{2} |z_n|}\left(|\vec{v}-
  \vec{v}_n|^2 + |\vec{v}-\vec{v}_n|^{\ell+2}\right)\,dx,\\
  &\leq& K\int_{|x|\geq \frac{1}{2} |z_n|} \left(|\vec{v}|^2 +
  |\vec{v}|^{\ell+2}\right)\,dx \\ &+& K\int_{\mathbb{R}^N}\left( |\vec{v}-\vec{v}_n|^2 +
  |\vec{v}-\vec{v}_n|^{\ell+2}\right)\,dx.
  \end{eqnarray*}
  Therefore, $$\lim_{n\rightarrow +\infty} |\int_{|x-z_n|\geq R_n} \left(F^\infty(x,\vec{v}_n)
  -F(x-z_n,\vec{v}_n) \right)\,dx|= 0.$$
  Thus, for all $\varepsilon >0$, we get $\liminf_{n\rightarrow +\infty}\{ \mathcal J(\vec{u}_n)-\mathcal J^\infty(\vec{u}_n)\} \geq -\varepsilon D$ and so $$\mathcal I_{c} = \lim_{n\rightarrow +\infty} \mathcal J(\vec{u}_n) \geq \liminf_{n\rightarrow +\infty} \mathcal J^\infty(\vec{u}_n)
  \geq \mathcal I^\infty_{c}.$$
  We reach a contradiction with the fact that $\mathcal I_{c} < \mathcal I_{c}^\infty$. Hence $\{y_n\}_{n\in\mathbb N}$ is bounded. Setting $\rho = \sup_{n \in \mathbb{N}}|y_n|$, it follows that for all $\varepsilon >0$
  \begin{eqnarray*}
  \int_{B(0,R(\varepsilon)+\rho)} \left(u^2_{n,1} +...+ u^2_{n,m}\right)\,dx &\geq&\int_{B(y_n,R(\varepsilon))} \left(u^2_{n,1}+...+u^2_{n,m}\right)\,dx,\\
  &\geq& c^2-\varepsilon.
  \end{eqnarray*}
Thus, for all $\varepsilon > 0$
\begin{eqnarray*}
\s\;|\vec{u}|^2\,dx &\geq& \int_{B(0,R(\varepsilon)+\rho)}
|\vec{u}|^2 \,dx, \\&=& \lim_{n\rightarrow +
\infty}\int_{B(0,R(\varepsilon)+\rho)}|\vec{u}_n|^2\,dx\geq c^2-\varepsilon.
\end{eqnarray*}
Hence $\displaystyle{\int}(u^2_1+..+u^2_m)\,dx \geq c^2$. On the other hand $\displaystyle{\int} (u^2_1 + u^2_2 +...+ u^2_m)\,dx \leq c^2$. Thus $\vec{u} \in \mathcal S_c$ and $\nr\vec{u}_n-\vec{u}\nr_{\vec{L}^2}
\stcn 0$. By the boundedness of $\{\vec{u}_n\}_{n\in\mathbb N}$ in $\vec{H}^1$, it follows that
$\vec{u}_n \stcn\vec{u}$ in $\vec{L}^p$ for all $p \in
[2,2^\star]$. Therefore $\lim_{n\rightarrow +\infty} \int
F(x,\vec{u}_n) \,dx= \int F(x,\vec{u})\,dx$ implying that $ \mathcal J(\vec{u})
= \mathcal I_{c}$. 
\section*{Appendix}
In this appendix, we present the concentration-compactness
Lemma in the multi-constrained setting for the reader convenience. Let $\{\vec{u}_n\}_{n\in\mathbb N}$ be a minimizing sequence of the problem $(\mathcal I)$, we
introduce its associate concentration function \[Q_n(R) =
\displaystyle{\sup_{y \in
\mathbb{R}^N}\int_{B_R+y}}\rho_n^2(\xi)d\xi,\] where $\rho^2_n(\xi) =
|\vec{u}_n|^2 = \displaystyle{\sum^m_{i=1}}u^2_{n,i}(\xi)$. Applying the concentration compactness method (see page 136-137 of Ref. \cite{PL1} and page 272-273 of Ref. \cite{PL2}), one of the following alternatives
occur :
\subsection*{Vanishing} That is, $\limsup_{ y \in \mathbb{R}^N}\displaystyle{\int_{y+B_R}}
|\vec{u}_n|^2 = 0$.
\vskip6pt
\subsection*{Dichotomy} That is, for all $1 \leq i \leq m$, there exists $a_i \in
(0,c_i)$ such that for all $\varepsilon > 0$, there exists  $n_0 \in
\mathbb{N}$ and two bounded sequences in $\vec{H}^1$ denoted by
$\{\vec{v}_n\}_{n\in\mathbb N}$ and $\{\vec{w}_n\}_{n\in\mathbb N}$ (all depending on $\varepsilon)$ such
that for all $n \geq n_0$, we have for all $1 \leq i \leq m$ and $p\in(2,2^\star]$
\begin{align*}
&\left|\s\; v^2_{n,i}\,dx - a^2_i\right| < \varepsilon,\;\; \left|\s w^2_{n,i}\,dx - (c^2_i - a_i^2)\right| <
\varepsilon, \\&\\
&\s\left(|\nabla \vec{u}_n|^2 - |\nabla \vec{v}_n|^2 -
|\nabla \vec{w}_n|^2\right)\,dx \geq -2\varepsilon, \\&\\
& \nr u_{n,i} - (v_{n,i}
+ w_{n,i})\nr_p \leq 4 \varepsilon.
\end{align*}
Furthermore, there exists a sequence $ \{y_n\}_{n\in\mathbb N} \subset \mathbb{R}^N$
and $\{R_n\}_{n\in\mathbb N} \subset (0,+\infty)$ such that $\displaystyle{\lim_{n\rightarrow + \infty}}R_n = + \infty$, $\textrm{dist}(\textrm{Supp}|v_{n,i}|, \textrm{Supp}|w_{n,i}|) \stcn +\infty$ and
\[\left\lbrace
\begin{array}{lcl}
v_{n,i} = u_{n,i} &\textrm{ if } &|x-y_n| \leq R_0,\\
|v_{n,i}| \leq |u_{n,i}| &\textrm{ if } &R_0 \leq |x-y_n| \leq 2R_0,\\
v_{n,i} = 0 &\textrm{ if } &|x-y_n|\geq 2R_0, \\
w_{n,i} = 0 &\textrm{ if } &|x-y_n| \leq R_n,\\
|w_{n,i}| \leq |u_{n,i}| &\textrm{ if } &R_n \leq |x-y_n| \leq 2R_n,\\
w_{n,i} = u_{n,i} &\textrm{ if } &|x-y_n|\geq 2R_n.
\end{array}
\right.\]
\subsection*{Compactness} That is, there exists a sequence $\{y_n\}_{n\in\mathbb N} \subset \mathbb{R}^N$ such that for
all $\varepsilon > 0$, there exists $R(\varepsilon) > 0$ such that
$$\int_{B(y_n,R(\varepsilon))} |\vec{u}_n|^2 \,dx\geq \sum^m_{i=1} c^2_i
- \varepsilon.$$ 
\vskip6pt
As suggested and stated by Lions in Ref. \cite{PL2}, page
137-138, to get the above properties, it suffices to apply his
method to $\rho_n$. Decomposing $\rho_n$ in the classical setting
and thus simultaneously $u_{n,i}$, leads to the properties of the
splitting sequences $\vec{v}_n$ and $\vec{w}_n$, mentioned above.
\subsection*{Acknowledgement} The first author thanks the Deanship of Scientific Research at
King Saud University for funding the work through the research group project No. RGP-VPP-124.

\end{document}